\documentclass[12pt, a4paper]{amsart}

\usepackage{a4wide}
\usepackage{bbm}
\usepackage{amsmath}
\usepackage{amsfonts}
\usepackage{amssymb}
\usepackage{MnSymbol}

\usepackage{hyperref}
\usepackage{relsize}
\usepackage{scrextend}
\usepackage{lineno}
\usepackage{enumerate}
 \numberwithin{equation}{section}
 \usepackage[shortlabels]{enumitem}
\usepackage{color}
\usepackage{amsaddr}

\usepackage{amsmath}
\usepackage{amssymb, amsfonts,amscd,verbatim,hyperref,cleveref,graphics}
\usepackage{xspace}

\DeclareMathOperator{\real}{Re}

\newtheorem{thm}{Theorem}[section]
\newtheorem{prop}[thm]{Proposition}
\newtheorem{cor}[thm]{Corollary}
\newtheorem{lemma}[thm]{Lemma}

\newtheorem{defn}[thm]{Definition}

\numberwithin{equation}{section}

\theoremstyle{definition}
\newtheorem{rem}[thm]{Remark}
\newtheorem{ex}[thm]{Example}

\def\NN{{\mathbb N}}

\def\RR{{\mathbb R}}

\def\CC{{\mathbb C}}
\def\C{{\textnormal {C}}}
\def\UC{{\textnormal {UC}}}
\def\Lip{{\textnormal {Lip}}}

\def\UCB{{\textnormal {UCB}}}

\DeclareSymbolFont{bbold}{U}{bbold}{m}{n}
\DeclareSymbolFontAlphabet{\mathbbold}{bbold}

\newcommand{\zs}

\newcommand{\goesru}{\xrightarrow{ru}}

\makeindex

\begin{document}

\title[Generation of relatively uniformly continuous semigroups]{Generation of relatively uniformly continuous semigroups on vector lattices}

\author{M.~Kaplin}

\address{Institute for Mathematics, Physics, and Mechanics,
Jadranska 19,
SI-1000 Ljubljana,
Slovenija\\ and \\ Faculty of Mathematics and Physics,
University of Ljubljana,
Jadranska ulica 19,
SI-1000 Ljubljana,
Slovenia}
\email{michael.kaplin@fmf.uni-lj.si}

\author{M.~Kramar Fijav\v z}
\address{Faculty for Civil and Geodetical Engineering,
University of Ljubljana,
Jamova cesta 2,
SI-1000 Ljubljana,
Slovenia \\ and \\ Institute for Mathematics, Physics, and Mechanics,
Jadranska ulica 19,
SI-1000 Ljubljana,
Slovenia}
\email{marjeta.kramar@fgg.uni-lj.si}

\thanks{The authors acknowledge financial support from the Slovenian Research Agency, Grant No. P1-0222.}

\begin{abstract}

In this paper we prove a Hille-Yosida type theorem for relatively uniformly continuous positive semigroups on vector lattices.
We introduce the notions of relatively uniformly continuous, differentiable, and integrable functions on $\RR_+$.
These notions allow us to study the generators of relatively uniformly continuous semigroups. 
Our main result provides sufficient and necessary conditions for an operator to be the generator of an exponentially order bounded, relatively uniformly continuous, positive semigroup.
\end{abstract}

\maketitle
\section{Introduction}
The Hille-Yosida Theorem is a milestone in the theory of one-parameter semigroups of operators and was originally proved in 1948 by E. Hille in \cite{Hille:48} and  K. Yosida \cite{Yosida:48}, independently. It enables the identification of a strongly continuous operator semigroup on a Banach space through the resolvents of its generator. Our main goal  here is to prove a counterpart of the Hille-Yosida Theorem for relatively uniformly continuous positive semigroups.

The presented paper can be viewed as a companion paper to \cite{Kandic:18} where the notion of a relatively uniformly continuous semigroup on vector lattices is introduced. This notion is motivated by various examples such as the heat semigroup, the (left) translation semigroup, and Koopman semigroups on  the space of continuous functions on the real line $\C(\RR)$ and on its sublattices such as the space of uniformly continuous functions $\UC(\RR)$, Lipschitz continuous functions $\Lip(\RR)$, and continuous functions with compact support  $\C_c(\RR)$. 
It is shown that many basic results from the strongly continuous operator semigroup theory on Banach spaces can be translated in an analogous way to this setting and foundations are laid for further studies.
We build upon these results and focus on generation properties of such semigroups.

The paper is structured as follows. In Section 2 we recall some basic notions and facts about relative uniform convergence. In Section 3 we introduce the notions of relatively uniformly continuous, differentiable, and integrable functions. We develop the appropriate calculus fitting to this context and show
a version of the Fundamental Theorem of Calculus.  In Section 4 we use these concepts to study the generators of relatively uniformly continuous positive semigroups. There we introduce the notions of an ru-closed and ru-densely defined operator on a vector lattice and show that every generator of a relatively uniformly continuous positive semigroup is such. The proofs presented here have similarities to  the $C_0$-semigroup case, see e.g. \cite{Engel:00}, however, due to convergence with respect to a regulator, dealing with ru-continuous semigroups is more difficult. While the orbit maps of strongly continuous semigroups on Banach spaces grow at most exponentially in norm, relatively uniformly continuous semigroups a priori do not experience such a behavior. Hence, we introduce the notion of exponentially order bounded semigroups.
In Section 5  we prove that the resolvent of the generator of such a semigroup is its Laplace transform and that it satisfies a certain property related to the exponential growth of the semigroup. The rest of this section is devoted to the proof of our main result, \Cref{Hille-Yosida}, using the so called Yosida approximations. We conclude by showing that every exponentially order bounded, relatively uniformly continuous, positive semigroup is uniquelly determined by its generator, see \Cref{uniqueness}.

\section{Preliminaries}

A net $(x_\alpha)_{\alpha} $ in a vector lattice $X$ \emph{converges relatively uniformly to $x \in X$} if one can find a (non-unique) \emph{regulator} $u\in X$ such that for each $\varepsilon>0$ there exists $\alpha_0 $ such that
$$|x_\alpha-x|\leq \varepsilon \cdot u\quad\text{ holds for all }\alpha\geq \alpha_0.$$ 
In this case we write $x_\alpha \goesru x$ (with respect to $u$) and $\textup{ru}-\underset{\alpha}{\lim} \  x_\alpha :=x$.  We call $x$ the \emph{relative uniform limit} (or \emph{ru-limit}, for short) of $(x_\alpha)_{\alpha} $.

A vector lattice is said to be \emph{Archimedean} if for each $x,y \in X$ from
 $ 0 \leq n x \leq  y $ for all $n \in \NN$ it follows that $x=0$.
Throughout this paper we denote by $X$ an Archimedean vector lattice.

The following properties for relatively convergent nets in $X$ are easy to verify; see e.g. \cite[Theorem 16.2]{Luxemburg:71}.
\begin{lemma}\label{ru convergence}
 Let $X$ be an Archimedean vector lattice.
\begin{enumerate}
\item [\textup{(i)}] If $(x_\alpha)_{\alpha}$ converges relatively
uniformly to $x$ as well as to $y$, then $x= y$.
\item [\textup{(ii)}] If $x_\alpha \goesru x$ with respect to $u_x$, $y_\alpha \goesru y$ with respect to $u_y$ and $a,b \in \RR$, then
\begin{itemize}
\item $ a x_\alpha+b  y_\alpha \goesru a x+b y$  with respect to  $ |a| u_x+|b| u_y,$
\item $x_\alpha \vee y_\alpha \goesru  x\vee y  $ with respect to $ u_x+u_y,$
\item $x_\alpha \wedge y_\alpha \goesru x\wedge y$ with respect to $ u_x+u_y,$
\item $x^+_\alpha \goesru x^+$ with respect to $u_x$,
\item $|x_\alpha| \goesru |x|$ with respect to $ u_x$, and 
\item if $x_\alpha$ is positive for all $\alpha$, then  $x$ is positive.
\end{itemize}
\end{enumerate}
\end{lemma}

It is evident that relative uniform convergence implies order convergence  and, by \cite[Ch.1 Proposition 5.9]{Peressini:67}, the converse is true for sequences if the vector lattice is $\sigma$-order complete and has the diagonal property.

For vector lattices $X$ and $Y$ a map $T \colon X \rightarrow Y$ \emph{preserves ru-convergence} if for every $x_\alpha \goesru x$ in $X$ one has  $Tx_\alpha\goesru Tx$ in $Y$. By \cite[Theorem 5.1]{Taylor:19}, a linear operator  between Archimedean vector lattices preserves ru-convergence if and only if it is order bounded.
In particular, if $T: X \mapsto X$ is a positive operator and $x_\alpha \goesru x$ with respect to a regulator $u$, then $Tx_\alpha \goesru Tx$ with respect to regulator $Tu$.

A subset $S$ of $X$ is called \emph{relatively uniformly closed}
whenever $(x_n)_{n \in \NN} \subset S$ and $x_n \goesru x$ imply $x \in S$. By \cite[Section 3]{Luxemburg:67}, the relatively  uniformly closed sets are exactly the closed sets of a certain topology in $X$, the \emph{relative uniform topology} which we denote by $\tau_{ru}$. 
The relative uniform topology has been first studied by W.A.J. Luxemburg and L.C. Moore in \cite{Luxemburg:67}; see also \cite{Moore:68}. 
It is evident that relative uniform convergence implies convergence in the relative uniform topology.

\begin{ex}\label{topology examples}
 On a vector lattice $X$ with an order unit $u \in X$ the relative uniform topology $\tau_{ru}$ is generated by the norm $$\|x\|_{u}:= \inf \{ \lambda >0 \ \colon \ |x| \leq \lambda \cdot u \},$$
since $x_\alpha \goesru x$ if and only if $x_\alpha \xrightarrow{\|\cdot\|_{u}} x$.  
It is well-known that in a completely metrizable locally solid vector lattice $(X, \tau)$ every convergent sequence has a subsequence which converges relatively uniformly to the same limit, see \cite[Lemma 2.30]{Aliprantis:07}. This immediately yields that a subset of $X$ is relatively uniformly closed if and only if it is $\tau$-closed, so that topologies $\tau_{ru}$ and $\tau$ agree. 
 In particular, if $X$ is a Banach lattice, then $\tau_{ru}$ agrees with norm topology. For $0<p<1$ the vector lattice $L^p(\RR)$ equipped with the topology $\tau$ induced by the metric $$\qquad d_p(f_1,f_2):=\int_{\RR} |f_1(x)-f_2(x)|^p \ dx$$
is a completely metrizable, locally solid vector lattice which is not locally convex. Hence, the relative uniform topology need not be locally convex in general.
\end{ex}

As mentioned above, a linear operator is order bounded if and only if it preserves ru-convergence. For convergence in the relative uniform topology we have the following result.
\begin{lemma}
Let $X$ and $Y$ be  Archimedean vector lattices. 
If  a linear operator $T:X \rightarrow Y$ preserves ru-convergence, then $T \colon (X, \tau_{ru}) \rightarrow (Y, \tau_{ru})$ is continuous.
\end{lemma}
\begin{proof}
It suffices to
 show that for a fixed relatively uniformly closed set $V \subset Y$ the set $T^{-1}(V)$ is relatively uniformly closed in $X$. Pick $(x_n)_{ n \in \NN} \subset T^{-1}(V)$ and $x \in X$ such that $x_n \goesru x$. By assumption,  $Tx_n \goesru Tx$ and, since $V$ is relatively uniformly closed in $Y$, we have $Tx \in V$, i.e., $x \in T^{-1}(V)$. 
\end{proof}
Note, that we even obtain an equivalence  in the lemma above if the space $Y$  has an order unit. 
In the rest of this paper we focus on relative uniform convergence.

We say that a sequence $(x_n)_{n \in \NN} \subset X$ is a \emph{relatively uniform Cauchy sequence} (or \emph{ru-Cauchy sequence}, for short) if one can find a regulator $u \in X$ such that for each $\varepsilon > 0$ there exists $N \in \NN$ such that
$| x_n - x_m| \leq\varepsilon \cdot u$ holds for all $n,m \geq N$.
We call $X$ \emph{relatively uniformly complete} (or \emph{ru-complete}, for short) if each relatively uniform Cauchy sequence in $X$ converges relatively uniformly in $X$. 

It is known that a vector lattice $X$ is ru-complete if and only if its every principal ideal is ru-complete and hence, also every ideal of $X$ is ru-complete; see e.g. \cite[Exercise 59.5]{Luxemburg:71}.

\begin{ex}\label{example for ruc vector lattices 1}
By \cite[Theorem 42.5]{Luxemburg:71}, every Dedekind complete vector lattice is ru-complete and hence,  for each $0< p< \infty$ the vector lattice $L^p(\RR)$ is ru-complete. 
By \cite[Theorem 43.1]{Luxemburg:71}, the vector lattice $\C(\RR)$ is ru-complete and hence, its ideals $\C_c(\RR)$ and the space of continuous functions vanishing at infinity $ \C_0(\RR)$ are also ru-complete. Furthermore, it is easy to see that the space of uniformly continuous bounded functions $\UCB(\RR)$ and the space of continuous  bounded functions $\C_b(\RR)$ are ru-complete.
\end{ex}
For the unexplained terminology and basic results on vector lattices and relative uniform convergence we refer to \cite[Ch. 9]{Luxemburg:71}, \cite[Sec. 1.5 and 4.1]{Peressini:67}  and \cite{Vulikh:67}.

\section{Relative uniform calculus}

In this section we introduce the concepts of continuity, differentiability, and integrability of a function from $ \mathbb R_+$ to $X$ in terms of relative uniform convergence and discuss their relationships.

A function $f  \colon \mathbb R_+\to X$ is called \emph{relatively uniformly continuous} (or \emph{ru-continuous}, or \emph{ruc}, for short) if one can find a \emph{continuity regulator} $u  \colon \mathbb R_+\to X$ such that for each $\varepsilon>0$ there exists $\delta >0$ such that
$$\left|f(h+t) -  f(t) \right| \leq \varepsilon \cdot u(t)$$
holds for all $t \geq 0$ and $h \in [-\min\{\delta,t\}, \delta]$. In this case we write
$$f(h+t) \goesru f(t) \ \text{ as } h\to 0 \quad\text{ or }\quad \textup{ru-}\lim_{h \to 0} f(h+t) = f(t).$$

A function $f  \colon \mathbb R_+\to X$ is called \emph{relatively uniformly differentiable}  (or \emph{ru-differentiable}, for short) if one can find a function $f' \colon \mathbb R_+\to X$ and a \emph{differentiation regulator} $u \colon \mathbb R_+\to X$ such that for 
each $\varepsilon>0$ there exists $\delta >0$ such that
$$\left|\dfrac{f(h+t)-f(t)}{h} - f'(t) \right| \leq \varepsilon \cdot u(t)$$
holds for all $t \geq 0$ and $h \in [-\min\{\delta,t\}, \delta]$.
 We call $f'$ the \emph{ru-derivative of $f$}. 

\begin{rem} \textup{(i)} By \Cref{ru convergence}, if $f   \colon \mathbb R_+\to X$ and $g \colon \mathbb R_+\to X$ are two ru-differentiable functions with ru-derivatives $f',g'$ and differentiation  regulators $u_f,u_g$, respectively, and $a,b \in \RR$, then the function $a  f+ b g$ is ru-differentiable with ru-derivative $a  f'+ b g'$ and differentiation  regulator $|a|  u_f + |b|  u_g$.\\
\noindent \textup{(ii)} If $X$ is a Banach lattice, then ru-differentiability implies differentiability with respect to the norm.
\end{rem}
\begin{prop}\label{diff to cont}
Every ru-differentiable function is also ru-continuous.
\end{prop}
\begin{proof}
Let $f \colon \mathbb R_+\to X$ be an ru-differentiable function with differentiation regulator  $u  \colon \mathbb R_+\to X$. Then  for 
each $\varepsilon>0$ there exists $0< \delta< \varepsilon$ such that
\begin{align*}
|f(h+t)-f(t) |  &\leq |h| \cdot \left|\dfrac{f(h+t)-f(t)}{h} - f'(t) \right| + |h| \cdot | f'(t) | \leq \varepsilon \cdot ( u(t) + |f'(t)|)
\end{align*} 
holds for all $t \geq 0$ and $h \in [-\min\{\delta,t\}, \delta]$.
\end{proof}

Let   $s \geq 0$. A function $f  \colon \RR_+ \rightarrow X$ is called \emph{relatively uniformly integrable on the interval $[0,s]$} if one can find  $I_s \in X$  and a regulator  $u_s \in X$ such that for each $\varepsilon >0$ there exists $\delta>0$ such that
$$\left | \sum_{i=1}^n (s_i-s_{i-1})f(t_i) - I_s\right| \leq \varepsilon \cdot u_s$$
holds for every  partition $\{s_0,s_1,\ldots,s_n\}$ of the interval $[0,s]$ with $\max_{1 \leq i \leq n}|s_i-s_{i-1}| \leq \delta $ and $t_i\in [s_{i-1},s_i]$, $1 \leq i \leq n$. Since $I_s$ is defined as an ru-limit, by \Cref{ru convergence}.(i), it is unique  and we write
$\int_0^s f(t) \ \textup{d}t  :=I_s$. We say that $f \colon \RR_+ \rightarrow X$ is \emph{relatively uniformly integrable} (or \emph{ru-integrable}, for short) if it is relatively uniformly integrable on the interval $[0,s]$ for all $s \geq 0$ and call the map $s \mapsto \int_0^s f(t) \ \textup{d}t $ the \emph{ru-integral} of $f$.

The following proposition states some important properties of ru-integrals which we shall use later on.

\begin{prop}\label{integral properties}
Let $f  \colon \mathbb R_+\to X$ and $g  \colon \mathbb R_+\to X$ be ru-integrable functions,  $a,b \in \RR$, $x,s \in \RR_+$, and $T$ a positive linear operator on $X$. Then the following assertions hold.
\begin{itemize}
\item [\textup{(i)}]  The function $a f+ b g$ is ru-integrable and we have $$\int_0^s \left ( a f(t)+ b g(t) \right ) \ \textup{d}t  = a  \int_0^s f(t) \ \textup{d}t \ + b  \int_0^s g(t) \ \textup{d}t .$$
\item [\textup{(ii)}] We have
\[ \int_0^s f(x+t) \ \textup{d}t = \int_0^{x+s} f(t) \ \textup{d}t - \int_0^x f(t) \ \textup{d}t .\]
\item [\textup{(iii)}] If $|f(t)| \leq g(t)$ for each $0 \leq t \leq s$, then $$ \left | \int_0^s f(t) \ \textup{d}t \right | \leq  \int_0^{s} g(t) \ \textup{d}t.$$
\item [\textup{(iv)}] We have \[T \int_0^s f(t) \ \textup{d}t \ =  \int_0^{s} Tf(t) \ \textup{d}t.\]
\end{itemize}
If, in addition, $$ \textup{ru-}\lim_{s \rightarrow \infty} \int_0^s  f(t) \ \textup{d}t = : \int_0^ \infty  f(t) \ \textup{d}t\quad \text{ and } \quad \textup{ru-}\lim_{s \rightarrow \infty} \int_0^s  g(t) \ \textup{d}t = : \int_0^\infty  g(t) \ \textup{d}t$$ exist, then the above results also hold for $s = \infty$. 
\end{prop}
\begin{proof} Assertion (i) follows directly from \Cref{ru convergence}.(ii).

In order to show (ii), take any partitions $\{s_0,s_1,\ldots,s_n\}$, $\{x_0,x_1,\ldots,x_m\}$ of the intervals $[0,s]$, $[0,x]$, respectively,  $t_i\in [s_{i-1},s_i]$ for $1 \leq i \leq n$, and $y_j\in [x_{j-1},x_j]$ for $1 \leq j \leq m$. Then, choosing $r_i:=x_i$ for $0 \leq i \leq m$ and $r_i:=x+s_i$ for $m+1 \leq i \leq m+n$ we obtain a partition  $\{r_0,r_1,\ldots,r_{m+n}\}$  of the interval $[0,x+s]$ and 
\begin{align*}
 \sum_{i=1}^n (s_i-s_{i-1})f(t_i+x)= \sum_{i=1}^{m+n} (r_i-r_{i-1})f(\tau_i) - \sum_{i=1}^m (x_i-x_{i-1})f(y_i) 
\end{align*}
where $\tau_i:=y_i$  for $1 \leq i \leq m$ and $\tau_i:=x+t_i$ for $m+1 \leq i \leq m+n$. This proves (ii).

We now verify assertion (iii). By assumption, there exist regulator functions $u_f, u_g  \colon \RR_+ \rightarrow X$ such that for each $\varepsilon >0$ and each appropriate partition $\{s_0,s_1,\ldots,s_n\}$ of the interval $[0,s]$ and $t_i\in [s_{i-1},s_i]$, $1 \leq i \leq n$, 
we have
$$\left | \sum_{i=1}^n (s_i-s_{i-1})f(t_i) - \int_0^s f(t) \ \textup{d}t\right| \leq \varepsilon \cdot u_f(s) \text{ and } \left | \sum_{i=1}^n (s_i-s_{i-1})g(t_i) - \int_0^s g(t) \ \textup{d}t\right| \leq \varepsilon \cdot u_g(s).$$
Hence, 
\begin{align*}
\left |\int_0^s f(t) \ \textup{d}t \right| &\leq \left | \int_0^s f(t) \ \textup{d}t - \sum_{i=1}^n (s_i-s_{i-1})f(t_i)\right|+ \sum_{i=1}^n (s_i-s_{i-1})|f(t_i)|\\
& \leq \varepsilon \cdot u_f(s) + \sum_{i=1}^n (s_i-s_{i-1})g(t_i) \leq \varepsilon \cdot (u_f(s)+ u_g(s)) +\int_0^s g(t) \ \textup{d}t.
\end{align*}
Since $X$ is Archimedean, we obtain (iii).

Assertion (iv) follows from the fact that positive operators preserve relative uniform limits.
\end{proof}

We now show a version of  the Fundamental Theorem of Calculus for ru-integrals and ru-derivatives. 

\begin{prop}
Let $f \colon \mathbb R_+\to X$ be an ru-continuous and ru-integrable function. Then the ru-integral of $f$ is ru-differentiable and its ru-derivative equals $f$. \end{prop}
\begin{proof}
By assumption, there exists a map $u \colon \RR_+ \rightarrow X$ such that for each $\varepsilon >0$ there exists $\delta >0$ such that $|f(t+s)-f(s)| \leq \varepsilon \cdot u(s)$ holds for all $s\geq 0$ and $t \in [-\min\{\delta, s\},\delta]$. 
Hence, by \Cref{integral properties}.(ii)-(iii),  we obtain
\begin{align*}
\left | \frac{\int_0^{s+h} f(t) \ \textup{d}t-\int_0^s f(t) \ \textup{d}t}{h} - f(s)\right |  & \leq \frac{1}{h}\left| \int_0^h (f(t+s)  - f(s))\ \textup{d}t\right | \leq \frac{1}{h} \int_0^h \varepsilon \cdot u(s) \ \textup{d}t =\varepsilon \cdot u(s)
\end{align*}
for all $s\geq 0$ and $h \in [-\min\{\delta, s\},\delta]$. 
\end{proof}

The following result will be used in the proof of \Cref{product rule for comm semigroups}. It is a version of the Newton-Leibniz theorem in the relatively uniform context.

\begin{prop}\label{fundamental theorem of calculus}
Let $f \colon \mathbb R_+\to X$ be ru-differentiable with differentiation  regulator $ u \colon \mathbb R_+\to X$ such that its ru-derivative $f'$ is ru-continuous with continuity  regulator $ \tilde u \colon \mathbb R_+\to X$. 
If $u$ and $\tilde u$ are ru-integrable, then $f'$ is ru-integrable and for each $s >0$ we have $$\int_0^s  f'(t) \ \textup{d}t=f(s)-f(0).$$
\end{prop}
\begin{proof}  By assumption, there exists $w\colon \RR_+ \rightarrow X$ such that for each $s \geq 0$ and $\varepsilon >0$ one can find $\delta_s>0$ such that for all partitions $\{s_0,s_1,\ldots,s_n\}$ of the interval $[0,s]$ with $\max_{1 \leq i \leq n}|s_i-s_{i-1}| \leq \delta_s $ and $t_i\in [s_{i-1},s_i]$, $1 \leq i \leq n$, we have
$$\left | \sum_{i=1}^n (s_i-s_{i-1})u(t_i) - \int_0^s u(t) \ \textup{d}t  \right| \leq  \varepsilon\cdot w(s) \text{ and } \left | \sum_{i=1}^n (s_i-s_{i-1})\tilde u(t_i) - \int_0^s \tilde u(t) \ \textup{d}t  \right| \leq \varepsilon \cdot w(s).$$
 Fix $s > 0$ and $\varepsilon >0$. By assumption, there exists $ 0< \delta < \delta_s$ such that $$ \left | \frac{f(h+t)-f(t)}{h} - f'(t)\right | \leq \varepsilon \cdot  u(t) \text{ and } \left | f'(h+t) - f'(t)\right | \leq \varepsilon \cdot  \tilde u(t) $$ hold for all $t \geq 0$ and $h \in [-\min\{ \delta,t\},  \delta]$. Now we estimate
\begin{align*}
 &\left | \sum_{i=1}^n (s_i-s_{i-1})f'(t_i) -(f (s)-f(0))  \right | 
 \leq  \sum_{i=1}^n (s_i-s_{i-1})\left |f'(t_{i}) -\frac{f (s_i )-f( s_{i-1})}{s_i-s_{i-1} }  \right | \\
&\leq \sum_{i=1}^n (s_i-s_{i-1})\left |f'(t_i) -f'(s_{i-1})  \right |+ \sum_{i=1}^n (s_i-s_{i-1})\left |f'(s_{i-1}) -\frac{f (s_i)-f( s_{i-1})}{s_i-s_{i-1} }  \right | \\
&\leq  \varepsilon \cdot \left (\sum_{i=1}^n (s_i-s_{i-1})\tilde u(s_{i-1})+\sum_{i=1}^n (s_i-s_{i-1})u(s_{i-1})   \right) \\
& \leq \varepsilon \cdot \left ( 2\varepsilon \cdot w(s) +  \int_0^s \tilde u(t) \ \textup{d}t+\int_0^s u(t) \ \textup{d}t  \right ). \qedhere
\end{align*}
\end{proof}

\section{Relatively uniformly continuous semigroups and generators}

As defined in \cite{Kandic:18}, a family $(T(t))_{t\geq 0}$ of linear operators on  $X$ is called a \emph{relatively uniformly continuous semigroup} (or \emph{ruc-semigroup}, for short) if it satisfies the following two conditions.
\begin{enumerate}[(i)]
\item  For each $ t,s\geq 0$ we have $T(s+t)=T(t)T(s)$ and $T(0)=I_X$.
\item For each $x \in X$  and $t \geq 0$ the orbit map $t \mapsto T(t)x$ is ru-continuous, i.e.,
$$T(h+t)x \goesru T(t)x \ \text{ as } h\to 0.$$
 \end{enumerate}
If, in addition,  $T(t)$ is a positive operator on $X$ for each $t \geq 0$, the semigroup  $(T(t))_{t\geq 0}$ is called \emph{positive}.

\begin{rem}
Since relative uniform convergence implies convergence in the relative uniform topology, the notion of relative uniform continuity allows us to study continuous semigroups on non-locally convex spaces such as $L^p(\RR)$ for $0<p<1$; see \Cref{topology examples}.
\end{rem}

It was shown in  \cite[Proposition 3.6]{Kandic:18} that  for a positive semigroup it suffices to check the ru-continuity of the orbit maps only at $t=0$ and for positive vectors, i.e.,
$$T(t)x \goesru  x  \ \text{ as }\ t \searrow  0  \text{ for }x\in X_+.$$
Another crucial property of ruc-semigroups is that orbit maps are order bounded on finite intervals; see \textbf{\cite[Proposition 3.5]{Kandic:18}}. For the general theory of positive operator semigroups and ruc-semigroups we refer to \cite{Batkai:17}, \cite{Engel:00} and \cite{Kandic:18}, respectively.

Next, we study the ru-integrability of the orbit maps of a positive ruc-semigroup  on an ru-complete vector lattice.
\begin{lemma}\label{uniformly complete + uniformly continuous}
Let $(T(t))_{t\geq 0}$ be a relatively uniformly continuous positive semigroup on a relatively uniformly complete vector lattice $X$. Then the following assertions hold for each $x \in X$ and $s\geq 0$. 
\begin{itemize}
\item [\textup{(i)}]  The orbit map $t \mapsto T(t)x$ is ru-integrable.
\item [\textup{(ii)}]  The operator $  y \mapsto \int_0^s T(\tau)y \ \textup{d}\tau $ on $X$ is well-defined and positive.
\item [\textup{(iii)}]  We have $y_h := \frac{1}{h} \left (\int_0^h T(\tau)x \ \textup{d}\tau \right ) \goesru x$ as $h \searrow 0$.
\end{itemize}
 \end{lemma}
\begin{proof}
To prove (i) fix $\varepsilon>0$. By assumption, there exist a positive element $u \in X$, independent of $\varepsilon$, and $\delta>0$ such that $|T(h)x-x|\leq \varepsilon \cdot u$ holds for all $h\in [0,\delta]$. Furthermore, by \cite[Proposition 3.5]{Kandic:18}, there exists $v\in X$ such that $T(t)(u\vee x)\leq v$ holds for all $t\in [0,s]$. In particular, for each $t \in [0,s]$ we have $T(t)x  \in I_{\{v\}}$, where $I_{\{v\}}$ is the order ideal generated by $v$. Pick $0\leq t' \leq t \leq s$ with $|t-t'|\leq \delta.$ Then
 $$|T(t)x-T(t')x|\leq T(t')|T(t-t')x-x|\leq \varepsilon \cdot T(t')u\leq \varepsilon \cdot v.$$
Hence, the mapping
 $$ \varphi \colon [0,s] \rightarrow I_{\{v\}}, \quad  t \mapsto T(t)x, $$
is continuous with respect to the AM-norm on $I_{\{v\}}$ defined by
$$\| y\|_{v}:= \inf \{\lambda >0 \ : \ |y|\leq \lambda \cdot v \}.$$
Since $X$ is Archimedean and ru-complete, the order ideal $I_{\{v\}}$ is complete with respect to the norm $\| \cdot\|_{v}$ and so there exists the unique Riemann integral in $X$ which is the ru-limit of the Riemann sums of the orbit map $t \mapsto T(t)x$ on $[0,s]$. 

To prove (ii) fix $y \in X_+$. We show that $\int_0^s T(\tau)y \ \textup{d}\tau \in X_+$. Indeed,  for each $t \geq 0$ the operator $T(t)$ is positive and thus, for any partition $\{s_0,s_1,\ldots,s_n\}$ of the interval $[0,s]$ and $t_i\in [s_{i-1},s_i]$, $1 \leq i \leq n$, the Riemann sum
$$\sum_{i=1}^n (s_i-s_{i-1})T(t_i)y$$
is positive in $X$. The element $\int_0^s T(\tau)y \ \textup{d}\tau$ is the ru-limit of a net of positive elements in $X$ and hence, $\int_0^s T(\tau)y \ \textup{d}\tau \in X_+$. 

To show (iii) fix $\varepsilon >0$. By assumption, there exist $u \in X$, independent of $\varepsilon $, and $\delta >0$ such that $|T(h)x -x| \leq \varepsilon \cdot u$ holds for all $h \in [0, \delta]$ and hence, by \Cref{integral properties}.(iii), we have
$$|y_h - x|=  \frac{1}{h} \left |\int_0^h (T(\tau)x-x) \ \textup{d}\tau \right|  \leq\varepsilon \cdot u$$
for all $h \in [0, \delta]$. 
\end{proof}

\begin{ex}\label{shift integrable on cont}
For a function $f \colon \ \RR \rightarrow \RR$ and $t \geq 0$, we consider the translation operator
$$(T_l(t)f)(x)=f(t+x), \quad x \in \RR.$$
 It is evident that by fixing a translation invariant space $Y$ of functions on $\RR$ one obtains a semigroup $(T_l(t))_{t \geq 0}$ on $Y$ which we call the \emph{(left) translation semigroup} on $Y$. 
 This semigroup is ru-continuous on $\UCB(\RR)$ which is an ru-complete vector lattice. Indeed, if we fix $f \in \UCB(\RR)$ and pick $\varepsilon >0$, then there exists $\delta >0$ such that $|f(h+t)-f(t)| \leq \varepsilon \cdot 1$ holds for all $h \in [0,\delta]$ and $t \in \RR$, and since the constant function is in $\UCB(\RR)$ we obtain the claim. Hence, by \cite[Proposition 3.1]{Kandic:18} and \cite[Corollary 5.9]{Kandic:18}, the (left) shift semigroup is relatively uniformly continuous on $\C_c(\RR)$ and $\C(\RR)$ which, by  \Cref{example for ruc vector lattices 1},  are ru-complete vector lattices.  So, $(T_l(t))_{t \geq 0}$ on $\C_c(\RR)$, $\C(\RR)$, and $\UCB(\RR)$ satisfies the assumptions of \Cref{uniformly complete + uniformly continuous}.
\end{ex}

Since we will repeatedly use  \Cref{uniformly complete + uniformly continuous},  in this section  $X$ will denote an ru-complete vector lattice.  Contrary to ru-integrability, the ru-differentiability of the orbit maps does not always hold. On the set of vectors, for which the orbits are ru-differentiable we can define the generator of an ruc-semigroup as follows.

The \emph{generator $A \colon D(A) \subset X \rightarrow X$ of a relatively uniformly continuous semigroup} $(T(t))_{t\geq 0}$ on $X$ is the operator
\begin{align*}
Ax := \textup{ru}-\underset{h \searrow 0}{\lim} \ \frac{1}{h}(T(h)x - x), \qquad
D(A) := \left \{x \in X  \ \middle \mid \ \textup{ru}-\underset{h \searrow 0}{\lim}\ \frac{1}{h}(T(h)x - x)\text{ exists in }X \right \}.
\end{align*}

\begin{rem}
Obviously, every ruc-semigroup determines its generator uniquely. \Cref{uniqueness} will show that  under additional assumptions the converse is also true.
\end{rem}

\begin{ex}\label{shift integrable on cont}
The generator of the (left) translation semigroup $(T_l(t))_{t \geq 0}$ on $\C_c(\RR)$ is the first derivative operator $A:=\dfrac{\textup{d}}{\textup{d}x}$ with the domain $$D(A)=\{ f \in \C_c(\RR)  \mid f \text{ is continuously differentiable} \}.$$ Indeed, if $(B, D(B))$ is the generator of $(T_l(t))_{t \geq 0}$, then, by definition, for fixed $f \in D(B)$ there exists $u \in \C_c(\RR)$ such that for each $\varepsilon >0$ there exists $0<\delta <1$ such that  we have
\begin{equation*}
\left |  \dfrac{f(h+x)-f(x )}{h}-(Bf)(x )\right | \leq \varepsilon \cdot u(x).
\end{equation*}
for all $x \in \RR$ and $h \in [0,\delta]$ and hence, $f$ is left differentiable with left derivative $Bf$. Since $Bf$ is a continuous function, $f $ is differentiable and $Bf =Af$. In particular, we have $D(B) \subset D(A)$. 

Now, let $f \in D(A)$. Then $Af \in \C_c(\RR)$ and hence, there exists $n \in \NN$ such that $f(x)=0$ and $Af(x)=0$ for all $x \in [-n,n]^c$. Furthermore, since $Af=f'$ is continuous, by \cite[Exercise 5.8]{Rudin:76}, for fixed $\varepsilon >0$ there exists $0 < \delta  < 1$ such that 
\begin{equation*}
\left |  \dfrac{f(h+x)-f(x )}{h}-f'(x )\right | \leq \varepsilon 
\end{equation*}
holds for all $h \in [0,\delta]$ and $x \in [-n,n]$. By Urysohn's lemma, there exists $u   \in \C_c(\RR)$ such that $u(x)=1$ holds for all $x \in [-n-1,n+1]$ and hence, we obtain
\begin{equation*}
\left |  \dfrac{(T_l(h)f)(x)-f(x )}{h}-(Af)(x )\right | = \left |  \dfrac{f(h+x)-f(x )}{h}-f'(x )\right | \leq  \varepsilon  \cdot u(x)
\end{equation*}
for all $h \in [0, \delta]$ and $x \in \RR$, i.e., $f \in D(B)$.
\end{ex}

The following lemma captures some of the important properties of generators of positive ruc-semigroups. It is motivated by properties from the classical theory of strongly continuous semigroups; cf.~\cite[II.1.3]{Engel:00}.
\begin{lemma}\label{generator}
Let $A$ be the generator of a relatively uniformly continuous positive semigroup $(T(t))_{t\geq 0}$ on an ru-complete vector lattice $X$. The following assertions hold for each $s \geq 0$.
\begin{itemize}
\item [\textup{(i)}] The operator $A \colon D(A) \subset X \rightarrow X$ is linear.
\item [\textup{(ii)}] For $x \in D(A)$ we have $T(s)x \in D(A)$ and $AT(s)x=T(s)Ax$. Futhermore, the orbit map $ t \mapsto T(t)x$ is ru-differentiable with ru-derivative $ t \mapsto T(t)Ax$. 
\item [\textup{(iii)}] For each $x \in X$ we have
$$ \int_0^s T(\tau)x \ \textup{d}\tau \in D(A).$$
\item [\textup{(iv)}] We have
\begin{equation*}
\begin{split}
T(s)x-x & = A \int_0^s T(\tau)x \ \textup{d}\tau \quad \text{ if } x \in X,\\
& =  \int_0^s T(\tau)Ax \ \textup{d}\tau \quad \text{ if } x \in D(A).
\end{split}
\end{equation*}
\end{itemize}
\end{lemma}
\begin{proof}
Assertion (i) follows directly from the linearity of the operators $T(t)$ and \Cref{ru convergence}. 
To prove (ii) fix $x \in D(A)$. By assumption, there exists $u \in X$ such that for each $\varepsilon >0$ there exists $\delta >0$ such that  $$\left |T(h)Ax - Ax \right | \leq \varepsilon \cdot u \text{ and }\left |\frac{1}{h}(T(h)x-x) - Ax \right | \leq \varepsilon \cdot u$$
holds for all $h \in [0 , \delta]$. Furthermore, by \cite[Proposition 3.5]{Kandic:18} there exists $v \in X$ such that $T(t)u \leq v$ holds for all $t \in[0, s]$. Therefore
\begin{align*}
\left |\frac{1}{h}\left (T(h)T(s)x-T(s)x\right) -T(s)Ax \right | 
& \leq T(s) \left |\frac{1}{h}(T(h)x-x) - Ax \right | \leq \varepsilon \cdot T(s)u \leq \varepsilon \cdot 2v
\end{align*} 
holds for all $h \in [0 , \delta]$.
Hence, we obtain $T(s)x \in D(A)$ and $AT(s)x=T(s)Ax$. Moreover,
\begin{align*}
\left |\frac{1}{h}\left (T(h)T(s)x-T(s)x\right) -T(s)Ax \right | &\leq T(h+s) \left (\left |\frac{1}{-h}(T(-h)x-x) - Ax \right | + \left |Ax -T(-h) Ax \right | \right) \\
&\leq \varepsilon \cdot T(h+s)( 2u)  \leq \varepsilon \cdot (2v)
\end{align*} 
 holds for all $h \in [-\min \{\delta ,s\},0]$. This proves that $ t \mapsto T(t)x$ is ru-differentiable with ru-derivative $ t \mapsto T(t)Ax$.

 To prove (iii) and (iv) fix $x \in X$. Using \Cref{integral properties}.(iv) and \Cref{integral properties}.(ii) twice we obtain 
\begin{equation*}
\begin{split}
\frac{1}{h}\bigg( T(h)&\int_0^t T(\tau)x \ \textup{d}\tau -\int_0^t T(\tau)x \ \textup{d}\tau \bigg) = \frac{1}{h}\bigg(\int_0^{t+h} T(\tau)x \ \textup{d}\tau-\int_0^h T(\tau)x \ \textup{d}\tau -\int_0^t T(\tau)x \ \textup{d}\tau \bigg)\\
& \overset{}{=}  \frac{1}{h} \int_0^h T(\tau)T(t)x \ \textup{d}\tau- \frac{1}{h} \int_0^h T(\tau)x \ \textup{d}\tau
\end{split}
\end{equation*} 
for each $h > 0$. By \Cref{uniformly complete + uniformly continuous}.(iii), the right-hand side converges relatively uniformly to $T(t)x-x$ as $h \searrow 0$. This proves (iii) and the first identity of (iv).
Furthermore, we have
$$\frac{1}{h}\bigg( T(h)\int_0^t T(\tau)x \ \textup{d}\tau -\int_0^t T(\tau)x \ \textup{d}\tau \bigg)=\int_0^t  T(\tau) \left (\frac{1}{h}( T(h)x - x)\right )\textup{d}\tau  $$
for each $h >0$. Since, by \Cref{uniformly complete + uniformly continuous}.(ii), the operator $ y \mapsto \int_0^s T(\tau)y \ \textup{d}\tau $ is positive on $X$  it preserves ru-convergence and, hence, the right-hand side converges relatively uniformly to $\int_0^t  T(\tau) Ax \ \textup{d}\tau$ as $h \searrow  0$. This proves the second identity of (iv).
\end{proof}
The generator  of  a  strongly  continuous semigroup on a Banach space is   closed and  densely defined; see e.g. \cite[II.1.4]{Engel:00}.  Before we state an analogue to this result in our setting we need to introduce the appropriate notions.

A set $D \subset X$ is called \emph{ru-dense} if for each $x \in X$ there exists a sequence $(x_n)_{n \in \NN} \subset D$ such that $x_n \goesru x$.
We call an operator $P$ on $X$ \emph{ru-densely defined} if its domain $D(P)$ is ru-dense in $X$. An operator $P$ on  $X$ with domain $D(P)$ is called \emph{ru-closed} if $x_n \goesru x$ and $Px_n \goesru y$ imply that $x \in D(P)$ and $Px=y$.

\begin{prop}\label{generator properties}
The generator of a positive relatively uniformly continuous semigroup is an ru-densely defined and ru-closed operator. 
\end{prop}
\begin{proof}
Consider a positive ruc-semigroup  $(T(t))_{t\geq 0}$ on $X$ with generator $A$. 
Take $x \in X$ and define $y_n:=n\int_0^{\frac{1}{n}} T(\tau)  x \ \textup{d}\tau$. By \Cref{generator}.(iii), $y_n \in D(A)$ for for each $n \in \NN$  and by \Cref{uniformly complete + uniformly continuous}.(iii), we  have $y_n\goesru x $ as $n \rightarrow \infty$. This proves that $A$ is ru-densely defined.

To show that $A$ is ru-closed pick $x, y\in X$ and $(x_n)_{n \in \NN} \subset D(A)$ such that $x_n \goesru x $ and  $Ax_n \goesru y $. By \Cref{generator}.(iv), the identity
$$T(h)x_n -x_n= \int_0^h T(\tau)  Ax_n \ \textup{d}\tau$$
holds for each $h >0$ and $n \in \NN$. Furthermore, since for each $h >0$ the operators $T(h)$ and $y \mapsto\int_0^h T(\tau)  y \ \textup{d}\tau$ preserve relative uniform convergence, we have $$T(h)x_n -x_n \goesru T(h)x-x  \quad \text{and} \quad \int_0^h T(\tau)  Ax_n \ \textup{d}\tau \goesru \int_0^h T(\tau)  y \ \textup{d}\tau$$ as $n \rightarrow \infty$. Hence, the identity
$$\frac{1}{h}(T(h)x -x)= \frac{1}{h}\int_0^h T(\tau)  y \ \textup{d}\tau$$
holds for each $h >0$. By \Cref{uniformly complete + uniformly continuous}.(iii), the right-hand side converges relatively uniformly to $y$ as $h \searrow 0$ and, hence $x \in D(A)$. Since at the same time the left-hand side converges relatively uniformly to $Ax$, we obtain $Ax=y$. This proves that $A$ is ru-closed.
\end{proof}

The following  result can be interpreted as the `product rule' for the ru-derivative of commuting semigroups and is vital for the proof of the main result of this paper, \Cref{Hille-Yosida}.

\begin{lemma}\label{product rule for comm semigroups}
Let $(T(t))_{t\geq 0}$ and $(S(t))_{t\geq 0}$ be relatively uniformly continuous positive semigroups on $X$ with generators $A$ and $B$, respectively. If $D(A) \subset D(B)$ and for each $s,t \geq0$ the operators $T(t)$ and $S(s)$ commute, then for each $x \in D(A)$ and $t \geq 0$ we have  $$T(t)x -S(t)x=\int_0^t T(t- \tau)S(\tau)(B-A)x   \ \textup{d}\tau.$$
\end{lemma}
\begin{proof}
Fix $x\in D(A)\subset D(B)$ and $t\ge 0$. We will prove that the function $$f \colon \tau \mapsto T(t- \tau)S(\tau)x.$$
is ru-differentiable  with the ru-derivative $$f' \colon \tau \mapsto T(t- \tau)S(\tau)(B-A)x,$$ which is ru-continuous,  and that there exists $w_t \in X$ such that the constant function $\tau \mapsto w_t$ is a differentiation regulator of $f$ and a continuity regulator of $f'$. If so, then $f$ satisfies the assumptions of \Cref{fundamental theorem of calculus} which yields the result.

By assumption, there exists $u \in X$ such that $$ \frac{T(h)x-x}{h} \goesru Ax, \quad \frac{S(h)x-x}{h} \goesru Bx,$$
 $$T(h)(B-A)x \goesru (B-A)x, \quad S(h)(B-A)x \goesru (B-A)x$$
with respect to a regulator $u$ as $h \searrow 0$. By \cite[Proposition 3.5]{Kandic:18}, there exists $v_t,w_t \in X$ such that $S(s)u \leq v_t$ and $T(s)v_t \leq w_t$ holds for all $s \in [0,t]$. 

Fix $\tau \in (0,t)$ and $\varepsilon >0$. Then for some $\delta >0$ we have 
\begin{align*}
&\left | \frac{f(\tau +h)-f(\tau)}{h}- f'(\tau)\right | \\
&=\left |\frac{T(t- \tau-h)S(\tau+h)x-T(t- \tau)S(\tau)x}{h} - T(t- \tau)S(\tau)(B-A)x \right | \\
&\leq T(t- \tau-h)S(\tau) \left | \frac{S(h)x-T(h)x}{h} - T(h)(B-A)x\right |\\
&\leq T(t- \tau-h)S(\tau) \left( \left | \frac{S(h)x-x}{h} -Bx \right | +  \left | \frac{T(h)x-x}{h} -Ax \right | +  | (B-A)x-T(h)(B-A)x | \right) \\
&\leq \varepsilon \cdot T(t- \tau-h)S(\tau) 3 u \leq \varepsilon \cdot 3 w_t
\end{align*}
for all $h \in [0, \min \{\delta, t-\tau\} ]$. Similarly,
\begin{align*}
&\left | \frac{f(\tau -h)-f(\tau)}{h}- f'(\tau)\right |  \leq \varepsilon \cdot 3 w_t
\end{align*}
holds for all $h \in [  0, \min \{\delta, \tau\}]$. This proves that $f$
is ru-differentiable on $[0,t]$ with ru-derivative $f' $ and that $\tau \mapsto 3w_t$ is a differentiation regulator of $f$. 

Furthermore, by using similar arguments, we obtain
\begin{align*}
&|f'(\tau +h)-f'(\tau)| = \left | T(t- \tau-h)S(\tau+h)(B-A)x- T(t- \tau)S(\tau)(B-A)x \right | \\
&\leq T(t- \tau-h)S(\tau) (| S(h)(B-A)x-(B-A)x| +|T(h)(B-A)x-(B-A)x |)\\
& \leq \varepsilon \cdot 2w_t  \leq \varepsilon \cdot 3w_t 
\end{align*}
for some  $\delta >0$ and all $h \in [0, \min \{\delta, t-\tau\} ]$, and
\begin{align*}
&|f'(\tau -h)-f'(\tau)| = \left | T(t- \tau+h)S(\tau-h)(B-A)x- T(t- \tau)S(\tau)(B-A)x \right | \leq \varepsilon \cdot 3w_t
\end{align*}
for all $h \in [  0, \min \{\delta, \tau\}]$. This proves that $f' $ is ru-continuous on $[0, t]$ with continuity regulator $\tau \mapsto 3w_t$ and hence, we conclude the result.
\end{proof}

It is well-known that every strongly continuous semigroup on a Banach space is exponentially bounded, see e.g.~\cite[Proposition I.5.5]{Engel:00}. We now define an analogous property for semigroups on vector lattices.

We call a semigroup $(T(t))_{t\geq 0}$ on $X$ \emph{exponentially order bounded} if there exists some $w \in \RR$ such that for each $x \in X$ there exists $u \in X$ such that for all $t \geq 0$ we have
$$|T(t)x| \leq \large{e^{w\cdot t}} u.$$
We call such an $w \in \RR$ an \emph{order exponent} of $(T(t))_{t\geq 0}$. 

\begin{ex}
 The multiplication semigroup $(T_q(t))_{t\geq 0}$, defined by
$$T_q(t)f=e^{q(\cdot) t  }f, \quad q \in \C_b(\RR)$$
for each $f  \colon \RR \rightarrow \CC $ and $t \geq 0$, is an  exponentially order bounded semigroup on $\C_c(\RR)$,  $\Lip(\RR)$, $\UC(\RR)$, $\UCB(\RR)$, $\C_b(\RR)$, $\C(\RR)$, and $L^p(\RR)$ $(0<p<\infty)$ with order exponent $\| q\|_\infty$, since
$$|T(t)f| \leq e^{\| q\|_\infty t }|f|.$$
\end{ex}

In general, a relatively uniformly continuous semigroup is exponentially order bounded only under some additional assumptions.

\begin{prop}\label{order unit}
 If a vector lattice $X$ has an order unit $u \in X$, then every relatively uniformly continuous positive semigroup $(T(t))_{t\geq 0}$ is exponentially order bounded. 
\end{prop} 
\begin{proof}
First, by assumption, there exists $\lambda >1$ such that   $T(1)u \leq \lambda  u$ hold. Fix $x \in X$. By \cite[Proposition 3.5]{Kandic:18}, there exists $v \in X$ such that $|T(s)x| \leq v$ for all $s \in [0,1]$.  Fix $t \geq 0$, $N \in \NN$, $0  \leq s <1$ such that $t=N +s$ and pick $\mu >0$ such that $v \leq \mu   u$. Then for $w:= \ln( \lambda)$ we have
 \[|T(t)x| \leq  T(N)  |T(s)x|  \leq T(N) v\leq \mu \cdot T(1)^N u \leq  \lambda ^N \cdot (\mu  u)  \leq e^{w  \cdot t} (\mu  u).\qedhere\]
\end{proof}

\begin{ex}
By \cite[Proposition 2.4]{Kandic:18}, the vector lattices $\Lip(\RR)$ and $\UC(\RR)$ have an order unit and, by \cite[Proposition 3.1]{Kandic:18}, the (left) translation semigroup $(T_l(t))_{t\geq 0}$ is a relatively uniformly continuous positive semigroup on $\Lip(\RR)$ and $\UC(\RR)$. Hence, by \Cref{order unit}, $(T_l(t))_{t\geq 0}$ is  exponentially order bounded on $\Lip(\RR)$ and $\UC(\RR)$.
\end{ex}
By \cite[Proposition 3.1]{Kandic:18} and \cite[Corollary 5.9]{Kandic:18}, the (left) translation semigroup is relatively uniformly continuous on $\C_c(\RR)$ and $ \C(\RR)$, but it is not  exponentially order bounded on these lattices as the next example shows.
\begin{ex} 
The (left) translation semigroup $(T_l(t))_{t\geq 0}$ is not  exponentially order bounded on the following spaces.
\begin{enumerate}[(a)]
\item On $\C_c(\RR)$: Fix a positive function $f \in \C_c(\RR)$ with $f(0)=1$ and assume that there exists $w \in \RR$ and $u\in \C_c(\RR)$ such that $T_l(t)f  \leq e^{w \cdot t} u$ holds for all $t \geq 0$. Then $1=f(0)=(T_l(t)f)(-t) \leq e^{-w \cdot t} u(-t)$ and, hence $u(-t) \geq e^{w \cdot t} >0  $ for all $t \geq 0$ which contradicts $u \in \C_c(\RR)$. 

\item On $\C(\RR)$: Consider the function $f \colon x \mapsto e^{x^2}$ and assume that there exist $w \in \RR$ and $u\in \C(\RR)$ such that $T_l(t)f  \leq e^{w \cdot t} u$ holds for all $t \geq 0$. Then $e^{t^2-w\cdot t} \leq u(0)$ for all $t \geq 0$ which is a contradiction.

\item On $L^p(\RR)$ for $0<p< \infty$: Consider the function $$f\colon x \mapsto \left|\frac{1}{x-\frac{1}{2}}\right|^{\frac{1}{2p}} \cdot \chi_{[0,1]}(x)$$ in $L^p(\RR)$. Assume that there exist $w \in \RR$ and $u\in L^p(\RR)$ such that $T_l(t)f  \leq e^{w \cdot t} u$, i.e., $$  e^{-w \cdot t}\left|\frac{1}{x+t-\frac{1}{2}}\right|^{\frac{1}{2p}}  \cdot \chi_{[0,1]}(x+t) \leq u(x)$$ holds for all $t \geq 0$ and almost every $x \in \RR$. 
Furthermore, for each $x \in \left [0, \frac{1}{2} \right ]$ there exists $t \in \left [0, \frac{1}{2} \right ]$ such that $x+t- \frac{1}{2} =0$ and hence, $u$ attains infinity a.e. on $\left [0, \frac{1}{2} \right ]$ which contradicts $u \in L^p(\RR)$.
\end{enumerate}
\end{ex}

For further studies we need to define the resolvent set and the resolvent operator in our setting. 
In order to do that it is necessary to consider vector lattices over complex fields. A \emph{complex vector lattice} $X_{\CC}$ is a complexification of an ru-complete vector lattice $X$ endowed with the modulus function $$|z|:= \sup_{0 \leq \theta < 2 \pi} |\cos(\theta)x+\sin(\theta)y|$$
which, by \cite[Lemma 3.1]{LuxemburgZaanen:71}, exists for each $z:=x+iy \in X_{\CC}$. It is well-known that all complex vector lattices are ru-complete.
For a better understanding of complex vector lattices we refer to \cite{LuxemburgZaanen:71}. 

Motivated by the fact that a strongly continuous semigroup on a Banach lattice is positive iff its generator is a resolvent positive operator (see \cite[Corollary 11.4]{Batkai:17}), we introduce the following notion.
For an operator $A$ on $X_{\CC}$ we define its \textit{positive resolvent set} by 
$$\rho_+ (A):= \{  \lambda \in \CC \colon \ R(\lambda, A):=(\lambda - A)^{-1} \text{ exists and is a positive operator on }X_{\CC} \}.$$ 
For each $w \in \RR$ set $\CC_{>w}:=\{  \lambda \in \CC \colon \ \real\lambda>w  \}$.

In \cite[Section 4]{Kandic:18} it was shown, that rescaling does not change the ru-continuity
of the semigroup. One can show even more.

\begin{lemma}\label{rescaled}
Let $(T(t))_{t\geq 0}$ be a positive ruc-semigroup on $X_{\CC}$ with generator $A$. 
Let $\mu\in\RR$ and $\alpha>0$. The rescaled semigroup $(S(t))_{t\geq 0}$ defined by
$$S(t):=e^{\mu t}T(\alpha t)$$ 
is again a positive ruc-semigroup with generator $B=\alpha A+ \mu I_X$, $D(B)=D(A)$ 
and resolvent $R(\lambda,B)= \frac{1}{\alpha} R(\frac{\lambda-\mu}{\alpha},A)$, for $\lambda\in\rho_+(B)$.
Moreover, if $(T(t))_{t\geq 0}$  is exponentially order bounded with order exponent $w$, then 
$(S(t))_{t\geq 0}$, is also  exponentially order bounded with order exponent $w+\mu$.

\end{lemma}

\begin{proof}
The claim that $(S(t))_{t\geq 0}$ is a positive ruc-semigroup follows directly from \cite[Section 4]{Kandic:18}. To see that $B$ is the generator of $(S(t))_{t\geq 0}$, fix $x \in D(A)$. For each $h >0$ we have 
 \begin{align*}
 \dfrac{e^{\mu h}T(\alpha h)x-x}{h} = \alpha e^{\mu h}   \dfrac{T(\alpha h)x-x}{\alpha h} +  \dfrac{e^{\mu h}-1}{h} \cdot x
\end{align*}
and hence, by assumption, the right-hand side converges relatively uniformly to  $\alpha A x+ \mu x$ as $h \searrow 0$. It is clear that $D(B)=D(A)$. For $\lambda\in\rho_+(B)$ we have $$\lambda-B = \alpha \cdot \left (\dfrac{\lambda- \mu}{\alpha} -A \right )$$
and hence, $R(\lambda,B)= \frac{1}{\alpha} R(\frac{\lambda-\mu}{\alpha},A)$ follows. Now, if there exists $w \in \RR$ such that for each $x \in X_{\CC}$ there exists $u \in X_{\CC}$ such that $|T(t)x| \leq e^{wt}u$
holds for all $t \geq 0$, then $|S(t)x| \leq e^{(\mu+w)t}u$ holds for all $t \geq 0$.
\end{proof}

\section{A Hille-Yosida-type generation theorem}

In this section we present and  prove the main result of the paper,  \Cref{Hille-Yosida}, which is an analogue to the classical Hille-Yosida Theorem (see \cite[II.3.5 Generation Theorem]{Engel:00}) for ruc-semigroups.
It provides a characterisation of those linear operators that are the generators of some exponentially order bounded relatively uniformly continuous positive semigroups. 
More precisely, the generators are characterised via the behaviour of their resolvents. 
Throughout this section $X$ denotes an ru-complete complex vector lattice.

The following result allows us to work with the resolvents of the  generators of exponentially order bounded  positive ruc-semigroups. It shows that these resolvents are the Laplace transforms of the corresponding semigroups.

\begin{prop}\label{resolvent properties}
Let $(T(t))_{t\geq 0}$ be a positive exponentially order bounded ruc-semigroup on $X$ with order exponent $w \in \RR$ and generator $A$. Then the following assertions hold.
\begin{itemize}
\item [\textup{(i)}] For each $\lambda \in \CC_{>w}$ the mapping $$x \mapsto R(\lambda)x:=  \int_0^\infty e^{-\lambda \cdot t  }T(t)x \ \textup{d}t$$ defines a positive linear operator on $X$.
\item  [\textup{(ii)}] For each $x \in X$ there exists $u \in X$ such that $$ |R(\lambda)^k x | \leq (\real \lambda -w)^{-k} \cdot u$$ holds for all $k \in \NN$ and $\lambda \in \CC_{>w}$. 
\item  [\textup{(iii)}] The positive resolvent set $\rho_+ (A)$ contains  $\CC_{>w}$ and $R(\lambda)=R(\lambda, A)$ holds for each $\lambda \in\CC_{>w}$.
\end{itemize}
\end{prop}
\begin{proof}
By assumption, there exists $w \in \RR$ such that for each $x \in X$, there exists $u \in X$ such that
$$|T(t)x| \leq e^{t w }u$$
holds for all $t \geq 0$. Hence, by \Cref{integral properties}.(ii)-(iii), for each $\real \lambda >w$, $S>s \geq 0$ and fixed $x \in X$ we obtain
\begin{align*}
 &\left |\int_0^S e^{- \lambda \cdot  t}T(t)x  \ \textup{d}t-  \int_0^s e^{-\lambda \cdot  t}T(t)x \ \textup{d}t \right |  = \left |\int_0^{S-s} e^{-\lambda \cdot  (s+t)}T(s+t)x \ \textup{d}t \right | \\
&  \leq  \int_0^{S-s} e^{-(\real \lambda-w) \cdot (s+t)} \cdot  u  \ \textup{d}t   \leq 2(\real \lambda-w)^{-1}  e^{-(\real \lambda-w) \cdot s}\cdot u.
\end{align*}
Since $X$ is relatively uniformly complete, the improper ru-integral defining $R(\lambda)x$ exists.
Furthermore, by \Cref{ru convergence}.(iii) and the fact that $T(t)$ is linear and positive for each $t \geq 0$, the operator $R(\lambda )$ is also linear and positive.  This proves (i).

To prove (ii) we use the assumption that $(T(t))_{t\geq 0}$ is exponentially order bounded with order exponent $w $ and \Cref{integral properties}.(iii)-(iv) $(n-1)$  times to estimate
\begin{align*}
|R(\lambda)^k x| &=  \left |\int_0^\infty \dots \int_0^\infty  e^{-\lambda \cdot \left (\sum_{l=1}^k t_l \right)}T\left (\sum_{l=1}^k t_l \right)x   \ \textup{d}t_1 \dots \textup{d}t_k \right|  \\
&\leq \int_0^\infty \dots \int_0^\infty  e^{-(\real\lambda -w) \cdot \left (\sum_{l=1}^k t_l \right)}  u \ \textup{d}t_1 \dots \textup{d}t_k  \\
& = \left ( \int_0^\infty e^{- (\real \lambda-w) \cdot t }  \ \textup{d}t  \right )^k \cdot u = ( \real \lambda - w)^{-k} \cdot u.
\end{align*}

Next, we show (iii). By a simple rescaling argument, see \Cref{rescaled}, we may assume that $\lambda=0$.  We need to show that $R(0,A)$ exists and equals $R(0)$. By \Cref{integral properties}.(ii), for each $h >0$ and $x\in X$ we have
\begin{align*}
\frac{T(h)-I}{h} R(0)x&=\frac{T(h)-I}{h} \int_0^\infty T(t)x  \ \textup{d}t = \frac{1}{h}\int_0^\infty T(t+h)x  \ \textup{d}t  - \frac{1}{h}\int_0^\infty T(t)x  \ \textup{d}t \\
& = - \frac{1}{h}\int_0^h T(t)x  \ \textup{d}t .
 \end{align*}
By \Cref{uniformly complete + uniformly continuous}.(ii), the right-hand side converges relatively uniformly to $-x$ as $h \searrow 0$ and therefore $ R(0)x \in D(A)$ with $AR(0)x=-x$ for all $x\in X$. On the other hand, for $x \in D(A)$, we obtain by \Cref{generator}.(iv) that
$$ AR(0)x = A \int_0^\infty T(t)x \ \textup{d}t \ = \int_0^\infty T(t) Ax \ \textup{d}t \ = R(0)Ax.$$
This proves (iii).
\end{proof}

The following property was introduced in \cite[Section 5]{Kandic:18}. 
It provides a substitute for the Principle of Uniform Boundedness,  which is an essential assumption in the rest of this paper.

\begin{defn}\label{property $(D)$ def}
A vector lattice $X$ has the property $(D)$ if for each net of linear operators $(T_\alpha)_{\alpha}$ on $X$ the following two assertions imply $T_\alpha x \goesru 0$ for each $x \in X$.
\begin{itemize}
\item [\textup{(a)}] There exists an $ru$-dense subset $D \subset X$ such that $T_\alpha y \xrightarrow{ru} 0$ for each $y \in D$.
\item [\textup{(b)}]  For each sequence $(x_n)_{n \in \NN} \subset X $ with $x_{n} \xrightarrow{ru} 0$ there exists $ u \in X_+$ such that for each $ \varepsilon >0$ there exist $N_\varepsilon \in \NN$ and $\alpha_\varepsilon $ such that $$|T_\alpha x_{n}| \leq \varepsilon \cdot u$$
holds for all $n \geq N_\varepsilon$ and $\alpha \geq\alpha_\varepsilon $.
\end{itemize}
\end{defn}

By \cite[Section 5]{Kandic:18}, the class of vector lattices which have the property $(D)$ contains all Banach lattices as well as $L^p(\RR)$ $(0<p<1)$, $\C(\RR)$, $\C_c(\RR)$, $\Lip(\RR)$,  $\UC(\RR)$, $\UCB(\RR)$ and $\C_b(\RR)$.

Let us immediately state a very useful criterion for the ru-continuity of an exponentially order bounded positive semigroup on an ru-complete vector lattice with the property $(D)$. It follows directly from \cite[Theorem 5.7]{Kandic:18}.
\begin{lemma}\label{ruc on dense to ruc on whole space}
 Let $X$ have the property $(D)$ and let $(T(t))_{t\geq 0}$ be an exponentially order bounded positive semigroup on $X$. If there exists an $ru$-dense set $D \subset X$ such that $T(h)y \goesru y$ as $h \searrow 0$ holds for each $y \in D$, then $(T(t))_{t\geq 0}$ is relatively uniformly continuous on $X$.
\end{lemma}

We are now ready to state our main generation result which is motivated by the classical theorems of Hille and Yosida.

\begin{thm}\label{Hille-Yosida}
Let $X$ be an ru-complete vector lattice with the property $(D)$ and $A$ a linear operator on $X$. Then the following assertions are equivalent.
\begin{itemize}
\item [\textup{(i)}]  Operator $A$ is the generator of an  exponentially order bounded relatively uniformly continuous positive semigroup  with order exponent 0.
\item [\textup{(ii)}]Operator $A$ is ru-closed, ru-densely defined, $\mathbb{C}_{>0} \subset \rho_+ (A)$ and for each $x \in X$ there exists $u \in X$ such that 
\begin{equation}\label{resolvent equ}
|R(\lambda,A)^k x | \leq (\real \lambda )^{-k} \cdot u
\end{equation}  
holds for all $k \in \NN$ and $\lambda \in \CC_{>0}$.
\end{itemize}
\end{thm}

While one of the implications in this theorem follows directly from \Cref{generator properties} and \Cref{resolvent properties}.(ii) more effort is needed for the proof of the other one. We start by showing a couple of lemmas. The operators $\lambda AR(\lambda,A)$ appearing in the following lemmas are known as \emph{Yosida approximants}.

\begin{lemma}\label{resolvent convergence}
Let $X$ have the property $(D)$ and let $A$ be an ru-closed and ru-densely defined operator on $X$ with  $ (0, \infty) \subset \rho_+ (A)$. Suppose that for each $x \in X$ there exists $u \in X$ such that 
\begin{equation*}
|R(\lambda,A) x | \leq  \lambda ^{-1} \cdot u
\end{equation*}  
 holds for all $\lambda >0$. Then the following assertions hold.
\begin{itemize}
\item [\textup{(i)}]  For each relatively uniformly convergent sequence $(x_n)_{n \in \NN} \subset X$ there exists $u \in X$ such that for each $\varepsilon >0$ there exists $N \in \NN$ such that $$\left|\lambda R(\lambda,A)x_n -x_n \right| \leq \varepsilon \cdot u$$ holds for all $\lambda , n \geq N$.
\item  [\textup{(ii)}]  For each relatively uniformly convergent sequence $(x_n)_{n \in \NN} \subset D(A)$ there exists $u \in X$ such that for each $\varepsilon >0$ there exists $N \in \NN$ such that $$\left|\lambda AR(\lambda,A)x_n -Ax_n \right| \leq \varepsilon \cdot u$$ holds for all $\lambda , n \geq N$.
\end{itemize}
\end{lemma}
\begin{proof}
To show (i) we prove first that $\lambda R(\lambda,A)x \goesru x$ as $\lambda \rightarrow \infty$ for each $x \in X$. Set $T_\lambda := \lambda R(\lambda,A)- I_X$ for each $\lambda > 0$. Since $X$ has the property $(D)$, it suffices to verify that $(T_\lambda)_{\lambda}$ satisfies assertions (a) and (b) from \Cref{property $(D)$ def}.
\begin{enumerate}[(a)]
\item By assumption, the set $D:=D(A)$ is ru-dense in $X$. For $x \in D$ we have $T_\lambda x= R(\lambda,A)Ax$
and hence, by assumption, there exists $u \in X$ such that
$|T_\lambda x| \leq \lambda^{-1} u$ holds for all $\lambda >0$ which yields $T_\lambda x \goesru 0$ as $\lambda \rightarrow \infty$.
\item Pick a sequence $(x_n)_{n \in \NN} \subset X $ such that $x_{n} \xrightarrow{ru} 0$ with respect to regulator $v \in X$. Fix $\varepsilon >0$. Then there exists $N_\varepsilon \in \NN$ such that $|x_n| \leq \varepsilon \cdot v $ holds for all $n \geq N_\varepsilon$.
By assumption, there exists $u \in X$ such that $ R(\lambda,A) v  \leq  \lambda^{-1} \cdot u$ holds for all $\lambda >0$ and since $R(\lambda,A)$ is positive for each $\lambda > 0$ we estimate
\begin{align*}
|T_\lambda x_n|& = |\lambda R(\lambda,A)x_n- x_n | \leq \lambda \cdot R(\lambda,A)|x_n|+ |x_n |
\leq \varepsilon \cdot (\lambda  R(\lambda,A) v + v) \\
& \leq \varepsilon \cdot ( u + v) 
\end{align*}
for all $\lambda >0$ and $n \geq N_\varepsilon$.
\end{enumerate}
By property $(D)$ we conclude that $\lambda R(\lambda,A)x \goesru x$ as $\lambda \rightarrow \infty$ for each $x \in X$. 

To finish the proof of (i) pick a sequence $(x_n)_{n \in \NN} \subset X$ and $x \in X$ with $x_n \goesru x$ with respect to regulator $u \in X$ as $n \rightarrow \infty$ and find  regulators $v_1, v_2 \in X$ such that $\lambda R(\lambda,A)x \goesru x$  with respect to  $v_1$ and $\lambda R(\lambda,A)u \goesru u$ with respect to  $v_2$ as $\lambda \rightarrow \infty$. Then for each $\varepsilon >0$ there exists $N \in \NN$ such that  $$|x_n-x| \leq \varepsilon \cdot u , \ \ \left|\lambda R(\lambda,A)x -x \right| \leq \varepsilon \cdot v_1, \text{ and }\left|\lambda R(\lambda,A)u -u \right| \leq \varepsilon \cdot v_2$$ for all $\lambda, n  \geq N$ and hence,
\begin{align*}
\left|\lambda R(\lambda,A)x_n -x_n \right| &\leq \left|\lambda R(\lambda,A)(x_n -x) \right|+ \left|\lambda R(\lambda,A)x -x \right|+\left| x-x_n \right|\\
& \leq \varepsilon \cdot \lambda R(\lambda,A) u + \varepsilon \cdot v_1 +  \varepsilon \cdot u \leq \varepsilon \cdot ( v_1 +\varepsilon v_2 + 2u).
\end{align*}
This proves (i). 
The second statement is an immediate
consequence of the first one.
\end{proof}

By using Yosida approximants we construct a sequence of exponentially order bounded relatively uniformly continuous positive semigroups which will play a crucial role in the proofs of \Cref{Hille-Yosida} and \Cref{uniqueness}.

\begin{lemma}\label{Approximation semigroups}
Let $X$ be an ru-complete vector lattice with the property $(D)$ and $A$ be as in \Cref{resolvent convergence}. Then for each $n \in \NN$ the operator
$$A_n := n^2 R(n,A) - n I_X = n AR(n,A)$$
on $X$ is the generator of the exponentially order bounded relatively uniformly continuous positive semigroup $(T_n(t))_{t\geq 0}$ with order exponent 0. Furthermore, these semigroups satisfy the following assertions.
\begin{itemize}
\item [\textup{(i)}]  For each $x \in X$ there exists $u \in X$ such that \begin{equation*}
|T_n(t)x| \leq u
\end{equation*}
holds for all $n \in \NN$ and $t \geq 0$.
\item [\textup{(ii)}] For each $x \in X$ there exists $u \in X$ such that for each $n \in \NN$ and $\varepsilon >0$ there exists $\delta >0$ such that 
 \begin{equation*}
\left |\frac{T_n(h)x-x}{h}-A_n x \right| \leq \varepsilon \cdot u
\end{equation*}
holds for all $h \in [0, \delta]$.
\item [\textup{(iii)}] The operators $T_n(t)$ and $T_m(s)$ commute for all $n ,m \in \NN$ and $t,s \geq 0$.
\end{itemize}
\end{lemma}
\begin{proof}
For each $n \in \NN$ we first define the operator $T_n(t)$ and show assertion (i). From this immediately follows that $(T_n(t))_{t\geq 0}$ is an exponentially order bounded positive semigroup with order exponent 0. Then we show assertion (ii) which also  yields that $A_n$ is the generator of the ruc-semigroup $(T_n(t))_{t\geq 0}$. At the end we verify assertion (iii).

Fix $x \in X$. By assumption \eqref{resolvent equ}, there exists $u \in X$ such that 
\begin{align}\label{jjj}
|(n  R(n,A))^k x | \leq  u
\end{align} 
 holds for all $n,k \in \NN$. 
 Then for each $t \geq 0$, $n \in \NN$ and all $\ell,m \in \NN$ with $\ell \geq m$ we estimate
\begin{align*}
\left | \sum\limits_{k=0}^{\ell} \dfrac{(tn)^k}{ k!}(nR(n,A))^k x -\sum\limits_{k=0}^m \dfrac{(tn)^k}{ k!}(nR(n,A))^k x \right| 
&= \left | \sum\limits_{k=m+1}^{\ell} \dfrac{(tn)^k}{ k!}(nR(n,A))^k x \right |  \\ &\leq  \sum\limits_{k=m+1}^{\ell} \dfrac{(tn)^k }{ k!} \cdot u.
\end{align*}
This shows that $\left ( \sum\limits_{k=0}^{\ell} \dfrac{t^k}{ k!}(n^2R(n,A))^k x \right)_{\ell \in \NN}$ is a relatively uniform Cauchy sequence in $X$ and hence, it has a unique limit  which we denote by $\sum\limits_{k=0}^\infty \dfrac{t^k}{ k!}(n^2R(n,A))^kx$ for each $t \geq 0$ and $n \in \NN$. 
Since $n^2R(n,A)$ is a positive linear operator the mapping $$T_n(t) \colon  y \mapsto e^{-nt} \sum\limits_{k=0}^\infty \dfrac{t^k}{ k!}(n^2R(n,A))^k y$$ defines a positive linear operator on $X$ for each $n \in \NN$ and $t \geq 0$.
 Furthermore, using \eqref{jjj} we estimate
\begin{equation}\label{absolutely}
\left | T_n(t)x \right|\leq e^{-nt}   \sum\limits_{k=0}^\infty \dfrac{(tn)^k}{ k!} \left |(nR(n,A))^k x  \right|  \leq e^{-nt}  \sum\limits_{k=0}^\infty \dfrac{(t n)^k }{ k!} \cdot u =  u
\end{equation}
for all $t \geq 0$ and $n \in \NN$. This proves (i).  Moreover, it follows that  $T_n(t)x$ is an element of the  principal ideal $I_{u}\subset X$ generated by $ u $ for all $t \geq 0$, $n \in \NN$ and since $X$ is ru-complete,  $I_{ u}$, endowed with the norm $$\| y \|_{ u}:=  \sup \{  \lambda >0   \ \colon  \  |y| \leq \lambda  u\}, $$
is a Banach lattice and, by \eqref{absolutely}, both series 
$$\sum\limits_{j=0}^\infty \dfrac{(tn)^j}{ j!}  \quad \textnormal{and} \quad \sum\limits_{k=0}^\infty \dfrac{t^k}{ k!}(n^2\|R(n,A)x\|_{ u})^k $$
converge absolutely. Hence, one can show, as  for the Cauchy product for scalar series, that for each $n \in \NN$ and $t \geq 0$, it holds
\begin{align}\label{formulaTn}
T_n(t)x&=\sum\limits_{k=0}^\infty \dfrac{t^k(-n)^k}{ k!} \cdot  \sum\limits_{k=0}^\infty \dfrac{t^k}{ k!}(n^2R(n,A))^k x  =  \sum\limits_{k=0}^\infty  \left (\sum\limits_{j=0}^k \dfrac{t^{k-j}(-n)^{k-j}}{(k-j)!} \cdot \dfrac{t^j}{ j!}(n^2R(n,A))^j x\right) \\
 &   =\sum\limits_{k=0}^\infty  \left ( \sum\limits_{j=0}^k\binom{k}{j}(-n)^{k-j}( n^2R(n,A) )^{j}x \right) \cdot t^k= \sum\limits_{k=0}^\infty  \dfrac{( n^2R(n,A) -n)^{k}x}{ k! }\cdot t^{k}= \sum\limits_{k=0}^\infty  \dfrac{ t^{k}}{ k! }A_n^{k}x.\nonumber
 \end{align}

Furthermore, using \eqref{formulaTn} and similar arguments as in \cite[Proposition I.2.3]{Engel:00}, it is easy to see that $(T_n(t))_{t\geq 0}$ defines a positive semigroup on $X$ for each $n \in \NN$.

We now show (ii). Using the  binomial formula and \eqref{jjj}, we obtain
 \begin{equation}\label{ank}
 |A_n^k x| = \left |\sum\limits_{j=0}^k\binom{k}{j}(-n)^j( n^2R(n,A) )^{k-j}x\right|\leq \left ( \sum\limits_{j=0}^k\binom{k}{j}n^j n^{k-j} \right )  \cdot u = (2n)^k \cdot u .
 \end{equation}
Now, fix $n \in \NN$ and $0< \varepsilon <1$. Then, by using \eqref{formulaTn} and \eqref{ank}, for all $h \in [0, \varepsilon \cdot e^{-2n}]$ we estimate
\begin{align*}
\left |\frac{T_n(h)x-x}{h}-A_n x \right| &=  h \cdot \left |  \sum\limits_{k=2}^\infty  \dfrac{h^{k-2} \cdot A_n^{k}x}{ k! } \right| \leq   h \cdot   \sum\limits_{k=2}^\infty  \dfrac{ \left |A_n^{k}x\right|}{ k! }  \leq   h \cdot    \left  (\sum\limits_{k=2}^\infty  \dfrac{(2n)^{k}}{ k! } \right) \cdot u \leq \varepsilon \cdot u.
\end{align*}
 This proves (ii) and shows that the orbit map $t \mapsto T_n(t)x$ is ru-differentiable and hence, by \Cref{diff to cont}, it is also ru-continuous, i.e., $(T_n(t))_{t\geq 0}$ is an ruc-semigroup on $X$.

Finally, assertion (iii) follows from formula \eqref{formulaTn} and the fact that $A_n$ and $A_m$ commute for all $n,m \in \NN$.
\end{proof}

We now proceed with the proof of the left implication of \Cref{Hille-Yosida}  which we  divide into three steps: 
\begin{itemize}
\item [{Step 1}:] By using the semigroups $(T_n(t))_{t \geq 0}$ defined in \Cref{Approximation semigroups}, for each $y \in D(A)$ and $t \geq 0$ we define $T(t)y$ as the ru-limit of $T_n(t)y$ as $n\to\infty$ and extend this definition to $X$.
\item [{Step 2}:] We show that $(T(t))_{t \geq 0}$ is an exponentially order bounded relatively uniformly continuous positive semigroup with order exponent 0.
\item [{Step 3}:] We prove that $A$ is the generator of $(T(t))_{t \geq 0}$.
\end{itemize}

\begin{proof}[Proof of \Cref{Hille-Yosida}, Step 1] Consider the Yosida approximants $A_n$ and the corresponding semigroups $(T_n(t))_{t \geq0}$ as defined in \Cref{Approximation semigroups}. Fix $x \in X$. Since $A$ is ru-densely defined, there exists a sequence $(x_k)_{k \in \NN} \subset D(A)$ such that $x_k \goesru x$. Take any such  sequence.

 We show first that for each $t \geq 0$ and $k \in \NN$ the sequence $(T_n(t)x_k)_{n \in \NN}$ is a relatively uniform Cauchy sequence.
By \Cref{resolvent convergence}.(ii), there exists $\tilde u \in X$ such that for each $\varepsilon >0$ there exists $N \in \NN$ such that
$$\left|A_n x_k -A_m x_k \right| \leq \left|n AR(n,A)x_k -Ax_k \right| + \left| Ax_k - m AR(m,A)x_k \right|
\leq \varepsilon \cdot (2 \tilde u)$$ holds for all $n , m,  k \geq N$. 
Furthermore, by \Cref{Approximation semigroups}.(i), there exist $\tilde w, v \in X$ such that for all $n \in \NN$, $t \geq 0$ we have
 $T_n(t) (2 \tilde u) \leq \tilde w $ and $T_n(t)\tilde w \leq v.$
Since $T_n(t)$ and $T_m(t)$ are positive operators, we estimate \begin{equation}\label{cauchy semigroup}
|T_m(t-\tau)T_n(\tau)(A_nx_k-A_mx_k)| \leq \varepsilon \cdot v \end{equation} 
for all $n , m,  k \geq N$ and $\tau \in [0, t]$. 
Moreover, by \Cref{Approximation semigroups}.(iii), for each $n,m \in \NN$ and $t,s \geq 0$ the operators $T_n(t),T_m(s)$ commute and hence, by \Cref{product rule for comm semigroups}, \eqref{cauchy semigroup}, and \Cref{integral properties}.(iii), for each $\varepsilon >0$ there exists $N \in \NN$ such that
\begin{equation*}
\begin{split}
|T_n(t)x_k-T_m(t)x_k|&= \left |\int_0^t  T_m(t-\tau)T_n(\tau)(A_nx_k-A_mx_k) \  \textup{d}\tau \right| \leq t \cdot \varepsilon \cdot v
\end{split}
\end{equation*}
holds for all  $n,m, k \geq N$ and $t \geq 0$. This proves that for each $k \geq N$ and  $t \geq 0$  the sequence $(T_n(t)x_k)_{n \in \NN}$ is a relatively uniform Cauchy sequence and hence, it has a limit which we denote by $ T(t)x_k$. Furthermore, there exists $v \in X$ such that for each $\varepsilon >0$ there exists $N \in \NN$ such that
\begin{equation}\label{approximation semigroup identity 2}
\begin{split}
|T_n(t)x_k-T(t)x_k| \leq t \cdot \varepsilon \cdot v
\end{split}
\end{equation}
for all  $n ,k\geq N$ and $t \geq 0$.
In particular, for each $t > 0$ and $\varepsilon >0$ there exists $\tilde N \in \NN$ such that 
\begin{equation}\label{approximation semigroup identity 1}
|T_n(t)x_k-T(t)x_k| \leq t \cdot \dfrac{\varepsilon}{t} \cdot v = \varepsilon \cdot v 
\end{equation}
holds for all $n,k \geq \tilde N$.

Next, we prove that $(T(t) x_k )_{k \in \NN}$ is a relatively uniform Cauchy sequence for each $ t \geq 0$. Assume that $x_k \goesru x$ with respect to a regulator $u$. By \Cref{Approximation semigroups}.(i), there exists $\tilde v \in X$ such that $T_n(t)u \leq \tilde v$ for all $n \in \NN$, $t \geq 0$ and hence, by \eqref{approximation semigroup identity 1}, for each $\varepsilon >0$ and  $t \geq 0$ there exists $ \tilde N \in \NN$ such that
\begin{align*}
 |T(t) x_k - T(t)x_m| &\leq |T(t) x_k - T_n(t)x_k| + T_n(t)| x_k - x_m|
 +  |T_n(t) x_m - T(t)x_m|\\
& \leq \varepsilon \cdot  ( v + T_n(t) u+   v) \leq \varepsilon \cdot  (2 v + \tilde v)
\end{align*}
holds for all $k,m \geq \tilde N$.
Hence, for each $t \geq 0$ the sequence $(T(t) x_k )_{k \in \NN}$ is a relatively uniform Cauchy sequence and it has a limit which we denote by $T(t)x$. Furthermore, there exists $\tilde{w} \in X$ such that for each $t \geq 0$ and $\varepsilon >0 $ there exists $ \tilde N \in \NN$ such that 
\begin{equation}\label{approximation semigroup identity 3}
|T(t) x_k-T(t)x| \leq \varepsilon \cdot \tilde{w}
\end{equation}
 holds for all  $k\geq  \tilde  N$.
As in the Banach space case, it is not difficult to verify that the limit $T(t)x$ is independent of the choice of $( x_k )_{k \in \NN}$.
\end{proof}
\begin{proof}[Proof of \Cref{Hille-Yosida}, Step 2] 
Since positivity and semigroup property are preserved under ru-limits,  $(T(t))_{t \geq 0}$ is a positive semigroup. We now show that it is exponentially order bounded with order exponent 0. To this end, fix $x \in X$ and pick any sequence $( x_k )_{k \in \NN} \subset D(A)$ such that $x_k \goesru x$ with respect to a regulator $u \in X$.  Then, by \eqref{approximation semigroup identity 3}, \eqref{approximation semigroup identity 1}  and \Cref{Approximation semigroups}.(i), there exists $v_1, v_2, v_3 \in X$ such that for each $t \geq 0$ there exists $N \in \NN$ such that 
$$ |T(t)x-T(t)x_N| \leq v_1 ,\quad|T(t)x_N -T_N(t)x_N| \leq v_2, \quad |x_N| \leq u +|x|, \quad  T_N(t)(u +|x|) \leq  v_3  $$ 
hold and hence, 
 \begin{align*}
 |T(t)x| &\leq |T(t)x-T(t)x_N| +|T(t)x_N -T_N(t)x_N| +T_N(t)|x_N|
 \\
  & \leq   v_1 + v_2+T_N(t)(u +|x|)  \leq v_1 + v_2+v_3.
 \end{align*} 
This proves that $(T(t))_{t \geq 0}$ is exponentially order bounded with order exponent 0.

It remains to prove that $(T(t))_{t \geq 0}$ is ru-continuous. By \Cref{ruc on dense to ruc on whole space}, it suffices to check that $T(h)y \goesru y$ as $h \searrow 0$ for each $y \in D(A)$. By the  same reasoning as in  the proof of  \eqref{approximation semigroup identity 2},  we  derive that there exists $w_1 \in X$ such that for fixed $ 0 < \tilde \varepsilon \leq 1$ there exists $\tilde N \in \NN$ such that $|T(h)y - T_{\tilde N}(h)y| \leq h \cdot \tilde \varepsilon \cdot w_1$ holds for all $h \geq 0$. Furthermore, since the semigroup $(T_{\tilde N}(t))_{t \geq 0}$ is ru-continuous there exists $  w_2 \in X$ such that for each $\varepsilon >0$ there exists $0< \delta< \varepsilon$ such that $ |T_{\tilde N}(h)y-y| \leq  \varepsilon \cdot   w_2$ holds for all $h \in [0 , \delta]$ and hence,
\[
\pushQED{\qed} 
 |T(h)y - y|  \leq |T(h)y - T_{\tilde N}(h)y| +|T_{\tilde N}(h)y-y|\leq h \cdot \tilde \varepsilon \cdot w_1 + \varepsilon \cdot  w_2  \leq \varepsilon \cdot (w_1+  w_2).  \qedhere
 \]   
\end{proof}
\begin{proof}[Proof of \Cref{Hille-Yosida}, Step 3] Let $B$ denote  the generator of $(T(t))_{t \geq 0}$. We show  that $A$ and $B$ coincide on $D(A)$ and that $ D(A) =D(B)$ which will conclude the proof. 

Fix $y \in D(A)$. As we mentioned in Step 2, from the proof of   \eqref{approximation semigroup identity 2} one can deduce that there exists $u_1 \in X$ such that  for each $\varepsilon >0$ there exist $N \in \NN$ such that
$$|T(h)y -T_N (h)y| \leq h \cdot \varepsilon \cdot u_1$$
holds for all $h \geq 0$ and $n \geq N$. Furthermore by \Cref{Approximation semigroups}.(ii) and \Cref{resolvent convergence}.(ii), there exist $u_2,u_3 \in X$ such that for each $\varepsilon >0$ there exist $M \geq N$ and $ \delta >0$ such that 
$$ \left |  \frac{T_M(h)y-y}{h}- A_M y\right |\leq \varepsilon \cdot u_2, \quad | A_M y -Ay  |\leq \varepsilon \cdot u_3$$
hold for all $h \in [0, \delta]$ and hence, we obtain
\begin{align*}
\left |   \frac{T(h)y-y}{h}- Ay\right | & \leq \left |   \frac{T(h)y -T_M (h)y}{h}\right |+\left |  \frac{T_M(h)y-y}{h}- A_M y\right | +| A_M y -Ay  | \\
& \leq \frac{h\cdot \varepsilon \cdot u_1}{h} + \varepsilon \cdot u_2 + \varepsilon \cdot u_3 \leq  \varepsilon \cdot (u_1+u_2+u_3).
\end{align*} 
This proves that $D(A) \subset D(B)$ and that $A$ coincides with $B$ on $D(A)$. 

To prove  $D(B) \subset D(A)$ fix $x \in D(B)$. Since $B$ is the generator of an exponentially order bounded semigroup with order exponent 0, by \Cref{resolvent properties}.(iii), we have $1 \in \rho_+(A)\cap \rho_+(B)$ and hence, $(I_X-A)$ and $(I_X-B)$ are bijective operators. Thus, there exists $y \in D(A)$ such that $(I_X-B)x =(I_X-A)y$. Since $(I_X-A)$ and $(I_X-B)$ coincide on $D(A)$ we obtain $(I_X-B)x =(I_X-B)y$ and hence, we have $x=y$. This proves that $x \in D(A)$.
\end{proof}

By applying \Cref{rescaled} we directly obtain a generalization of \Cref{Hille-Yosida} for exponentially order bounded ruc-semigroups of any order exponent.
\begin{cor}
Let $X$ be an ru-complete vector lattice with the property $(D)$.
For $w \in \RR$ the following assertions are equivalent.
\begin{itemize}
\item [\textup{(i)}] The operator $A$ is the generator of an  exponentially order bounded relatively uniformly continuous positive semigroup with order exponent $w $.
\item [\textup{(ii)}] The operator $A$ is ru-closed, ru-densely defined, $\mathbb{C}_{>w} \subset \rho_+ (A)$ and for each $x \in X$ there exists $u \in X$ such that $$ |R(\lambda,A)^k x | \leq (\real \lambda -w)^{-k} \cdot u$$ holds for all $k \in \NN$ and $\lambda \in \CC_{>w}$.
\end{itemize}
\end{cor}

We conclude by showing that every exponentially order bounded positive ruc-semigroup is uniquely determined by its generator.

\begin{prop}\label{uniqueness}
Let $X$ be an ru-complete vector lattice with the property $(D)$.
Every exponentially order bounded relatively uniformly continuous positive semigroup on $X$ is uniquely determined by its generator.
\end{prop}
\begin{proof}
By a simple rescaling argument, see \Cref{rescaled}, 
we may assume that $(S(t))_{t\geq 0}$ is an exponentially order bounded positive ruc-semigroup with order exponent 0. We will prove that $(S(t))_{t\geq 0}$ coincides with the semigroup $(T(t))_{t\geq 0}$ which was constructed in Step 1 of the proof of the backward implication in \Cref{Hille-Yosida}. 

Assume that $A$ is the generator of $(S(t))_{t\geq 0}$. By \Cref{resolvent properties}, the resolvent set $\rho_+ (A)$ contains  $\CC_{>0}$ and we have 
\begin{equation}\label{commute with resolvent}
R(n, A)x= \int_0^\infty e^{-n \cdot t  }S(t)x \ \textup{d}t
\end{equation}
for each $n \in \NN$ and $x \in X$. Furthermore, by \Cref{resolvent properties}.(ii) and \Cref{generator properties}, $A$ satisfies the assumptions of \Cref{resolvent convergence} and hence, by \Cref{Approximation semigroups}, for each $n \in \NN$
 the operator
$$A_n := n^2 R(n,A) - n I_X = n AR(n,A)$$
on $X$ is the generator of the exponentially order bounded  positive ruc-semigroup $(T_n(t))_{t\geq 0}$ with order exponent 0. 

Fix $y \in D(A)$. By \Cref{resolvent convergence}.(ii), there exists $u \in X$ such that for each $\varepsilon >0$ there exists $N \in \NN$ such that $$\left|A_ny -Ay \right| \leq \varepsilon \cdot u$$ for all $n  \geq N$. Furthermore, by assumption and \Cref{Approximation semigroups}.(i), there exist $w, v \in X$ such that $S(t)w \leq v$ and $T_n(t)u \leq w$  
for all $n \in \NN$, $t \geq 0$. Hence, for each $t \geq0$ we have \begin{equation}\label{cauchy semigroup 2}
|S(t-\tau)T_n(\tau)(A_ny-Ay)| \leq \varepsilon \cdot v \end{equation} 
for all $n \geq N$ and $\tau \in [0, t]$. 

By identity \eqref{commute with resolvent}, the operators $S(t)$  and $A_n$ commute for each $n \in \NN$, $t \geq 0$ and hence, by \eqref{formulaTn},  the operators  $S(t)$ and $T_n(s)$ commute for each $t, s \geq 0$ and $n \in \NN$. Therefore, by \Cref{product rule for comm semigroups}, \eqref{cauchy semigroup 2}, and \Cref{integral properties}.(iii), for each $\varepsilon >0$ there exists $N \in \NN$ such that
\begin{equation*}
\begin{split}
|S(t)y-T_n(t)y|&= \left |\int_0^t  S(t-\tau)T_n(\tau)(A_n y-A y) \  \textup{d}\tau \right| \leq t \cdot \varepsilon \cdot v
\end{split}
\end{equation*}
holds for all $n\geq N$ and $t \geq 0$.
This proves that $T_n(t)y \goesru S(t)y$ as $n \rightarrow \infty$ and hence, $S(t)y=T(t)y$  for each $t \geq 0 $ and $y \in D(A)$. Since $D(A)$ is ru-dense in $X$ and $S(t)$ and $T(t)$ preserve relative uniform limits we obtain $S(t)x=T(t)x$ for every $x\in X$ and $t \geq 0 $.
\end{proof}

\section*{Acknowledgments} 
We would like to express our gratitude to Marko Kandi\' c for insightful  comments and meticulous proofreading which improved the manuscript. The first author would also like to thank Marko for professional guidance.

\bibliographystyle{abbrv}

\begin{thebibliography}{10}

\bibitem{Aliprantis:07}
C.~D. Aliprantis and R.~Tourky.
\newblock {\em Cones and duality}, volume~84 of {\em Graduate Studies in
  Mathematics}.
\newblock American Mathematical Society, Providence, RI, 2007.

\bibitem{Batkai:17}
A.~B\'atkai, M.~{Kramar~Fijav\v{z}}, and A.~Rhandi.
\newblock {\em Positive operator semigroups: From finite to infinite
  dimensions}, volume 257 of {\em Operator Theory: Advances and Applications}.
\newblock Birkh\"auser/Springer, Cham, 2017.

\bibitem{Engel:00}
K.-J. Engel and R.~Nagel.
\newblock {\em One-parameter semigroups for linear evolution equations}, volume
  194 of {\em Graduate Texts in Mathematics}.
\newblock Springer-Verlag, New York, 2000.

\bibitem{Hille:48}
E.~Hille.
\newblock {\em Functional {A}nalysis and {S}emi-{G}roups}.
\newblock American Mathematical Society Colloquium Publications, vol. 31.
  American Mathematical Society, New York, 1948.

\bibitem{Kandic:18}
M.~Kandi\'c and M.~Kaplin.
\newblock Relatively uniformly continuous semigroups on vector lattices.
\newblock 2018.
\newblock arXiv:1807.02543.

\bibitem{Luxemburg:67}
W.~A.~J. Luxemburg and L.~C. Moore, Jr.
\newblock Archimedean quotient {R}iesz spaces.
\newblock {\em Duke Math. J.}, 34:725--739, 1967.

\bibitem{LuxemburgZaanen:71}
W.~A.~J. Luxemburg and A.~C. Zaanen.
\newblock The linear modulus of an order bounded linear transformation. {I}.
\newblock {\em Indagationes Mathematicae (Proceedings)}, 74, 1971.

\bibitem{Luxemburg:71}
W.~A.~J. Luxemburg and A.~C. Zaanen.
\newblock {\em Riesz spaces. {V}ol. {I}}.
\newblock North-Holland Publishing Co., Amsterdam-London; American Elsevier
  Publishing Co., New York, 1971.
\newblock North-Holland Mathematical Library.

\bibitem{Moore:68}
L.~C. Moore, Jr.
\newblock The relative uniform topology in {R}iesz spaces.
\newblock {\em Nederl. Akad. Wetensch. Proc. Ser. A 71 = Indag. Math.},
  30:442--447, 1968.

\bibitem{Peressini:67}
A.~L. Peressini.
\newblock {\em Ordered topological vector spaces}.
\newblock Harper \& Row, Publishers, New York-London, 1967.

\bibitem{Rudin:76}
W.~Rudin.
\newblock {\em Principles of mathematical analysis}.
\newblock International series in pure and applied mathematics. McGraw-Hill,
  3rd ed edition, 1976.

\bibitem{Taylor:19}
M.~A. Taylor and V.~G. Troitsky.
\newblock Bibasic sequences in {B}anach lattices.
\newblock 2019.
\newblock arXiv:1907.07589.

\bibitem{Vulikh:67}
B.~Z. Vulikh.
\newblock {\em Introduction to the theory of partially ordered spaces}.
\newblock Translated from the Russian by Leo F. Boron, with the editorial
  collaboration of Adriaan C. Zaanen and Kiyoshi Is\'eki. Wolters-Noordhoff
  Scientific Publications, Ltd., Groningen, 1967.

\bibitem{Yosida:48}
K.~Yosida.
\newblock On the differentiability and the representation of one-parameter
  semi-group of linear operators.
\newblock {\em J. Math. Soc. Japan}, 1:15--21, 1948.

\end{thebibliography}

\end{document}